\documentclass{article}  
\usepackage[final]{graphicx}   
\usepackage{psfrag}   
\usepackage{amsmath,amsfonts,amssymb,amsxtra,subeqnarray}

\newcommand {\Real}{\ensuremath{{\mathbb{R}}}}
\newcommand {\Natural}{\ensuremath{{\mathbb{N}}}}
\newcommand {\Complex}{\ensuremath{{\mathbb{C}}}}

\newcommand{\setO}{\ensuremath{\mathcal O}}
\newcommand{\setP}{\ensuremath{\mathcal P}}

\newcommand{\A}{\ensuremath{\mathcal A}}

\newcommand{\setL}{\ensuremath{\mathcal L}}
\newcommand{\setS}{\ensuremath{\mathcal S}}

\newcommand{\setE}{\ensuremath{\mathcal E}}
\newcommand{\X}{\ensuremath{\mathcal X}}
\newcommand{\G}{\ensuremath{\mathcal G}}
\newcommand{\N}{\ensuremath{\mathcal N}}
\newcommand{\W}{\ensuremath{\mathcal W}}

\newcommand{\vi}{\ensuremath{{\mathbf{v}}}}
\newcommand{\ex}{\ensuremath{{\mathbf{x}}}}

\newcommand{\one}{\ensuremath{{\mathbf{1}}}}

\newcommand{\dabilyu}{\ensuremath{{\mathbf{w}}}}

\newtheorem{theorem}{Theorem}

\newtheorem{corollary}{Corollary}

\newtheorem{lemma}{Lemma}
\newtheorem{fact}{Fact}
\newtheorem{algorithm}{Algorithm}
\newtheorem{claim}{Claim}
\newtheorem{definition}{Definition}

\newtheorem{proposition}{Proposition}
\newenvironment{proof}{\noindent {\bf Proof.}}{\hfill \hspace*{1pt}\hfill$\blacksquare$}

\begin{document}
\title{Conditions for synchronizability in arrays of coupled linear systems}
\author{S. Emre Tuna\\
{\small {\tt tuna@eee.metu.edu.tr}} }
\maketitle

\begin{abstract}
Synchronization control in arrays of identical output-coupled
continuous-time linear systems is studied. Sufficiency of new
conditions for the existence of a synchronizing feedback law are
analyzed. It is shown that for neutrally stable systems that are
detectable form their outputs, a linear feedback law exists under
which any number of coupled systems synchronize provided that the
(directed, weighted) graph describing the interconnection is fixed and
connected. An algorithm generating one such feedback law is
presented. It is also shown that for critically unstable systems
detectability is not sufficient, whereas full-state coupling is, for
the existence of a linear feedback law that is synchronizing for
all connected coupling configurations.
\end{abstract}

\section{Introduction}
Synchronization is, at one hand, a desired behaviour in many dynamical
systems related to numerous technological applications
\cite{fabiny93,fax04,cortes04}; and, at the other, a
frequently-encountered phenomenon in biology
\cite{stoop00,walker69,aldridge76}. Aside from that, it is an
important system theoretical topic on its own right.  For instance,
there is nothing keeping us from seeing a simple Luenberger observer
\cite{luenberger64} with decaying error dynamics as a two-agent system
where the agents globally synchronize \cite[Ch. 5]{fradkov07}. Due to
abundance of motivation to investigate synchronization, a wealth of
literature has been formed, mostly in recent times
\cite[Sec. 5]{boccaletti06}, \cite{strogatz01,wang02}.

An essential problem in control theory is to find general conditions
that imply synchronization of a number of coupled individual
systems. This problem, which is usually studied under the name {\em
synchronization stability}, has been attacked by many and from various
angles. Two cases, which partly overlap, are of particular interest:
(i) where the dynamics of individual systems are relatively primitive
(such as that of an integrator) yet the coupling between them
satisfies but minimum regularity stipulations (e.g. time-varying,
delayed, directed, need not be connected at all times); and (ii) where
the individual systems are let be more complex, but the structure of
their interaction resides in a more stringent set (e.g. fixed,
balanced, symmetric, with coupling strength greater than some
threshold.) The studies concentrated on the first case have resulted
in the emergence of the area now known as consensus in multi-agent
systems
\cite{olfati07,moreau05,cortes07,jadbabaie03,
ren05,angeli06,blondel05,tsitsiklis86} where fairly weak conditions on
the interconnection have been established under which the states of
individual systems converge to a common point fixed in space. The
second case has also accommodated important theoretical developments
especially by the use of tools from algebraic graph theory
\cite{wu05con}. Using Lyapunov functions, it has been shown that
spectrum of the coupling matrix plays a crucial role in determining
the stability of synchronization
\cite{wu95,pecora98} notwithstanding it need not be
explicitly known \cite{belykh06}. It has also been shown that
passivity theory can be useful in studying stability provided that the
coupling is symmetric
\cite{arcak07,pogromsky01,stan07}.

As alluded to earlier, Luenberger observer makes a particular example 
for synchronization. The well-known dynamics read
\begin{eqnarray*}
\dot{x}&=&Ax\\
\dot{\hat{x}}&=&A\hat{x}+L(y-C\hat{x})
\end{eqnarray*}
where $y=Cx$ is the output of the observed system $\dot{x}=Ax$ and
$\hat{x}$ is the estimate of the actual state $x$. Taking observer
gain as a design parameter, detectability of the pair $(C,\,A)$ is
necessary and sufficient for the existence of a linear feedback law
$L$ that ensures $|x(t)-\hat{x}(t)|\to 0$ as $t\to\infty$. Once the
above dynamics are rewritten as
\begin{eqnarray*}
\left[\!\!
\begin{array}{c}
\dot{x}\\ \dot{\hat{x}}
\end{array}\!\!
\right]=\left(I\otimes A+
\left[\!\!
\begin{array}{rr}
0&0\\1&-1
\end{array}\!\!
\right]
\otimes LC\right)
\left[\!\!
\begin{array}{c}
x\\ \hat{x}
\end{array}\!\!
\right]
\end{eqnarray*}
a natural generalization of the (observer design) problem
becomes apparent. Namely, for the following system
\begin{eqnarray*}
\left[\!\!
\begin{array}{c}
{\dot{x}}_{1}\\ \vdots\\{\dot{x}}_{p}
\end{array}\!\!
\right]=\left(I\otimes A+
\Gamma
\otimes LC\right)
\left[\!\!
\begin{array}{c}
{{x}}_{1}\\ \vdots\\ {{x}}_{p}
\end{array}\!\!
\right]
\end{eqnarray*}
investigate conditions on (coupling) matrix $\Gamma$ and pair
$(C,\,A)$ under which one can find an $L$ (synchronizing feedback law)
yielding $|x_{i}(t)-x_{j}(t)|\to 0$ for all $i,\,j$ (synchronization.)
Most applications related to this question render the assumption that
coupling matrix is known too restrictive. Even the size of $\Gamma$
(hence the number of coupled systems) may sometimes be
unknown. Therefore when designing a synchronizing feedback law, one is
forced to not count on the exact knowledge of coupling matrix
(interconnection.) However, something has to be assumed and how loose
that assumption can get is yet a partly open problem. To be precise,
given a set of possible interconnections, where each element of the
set determines not only the coupling between the systems but also the
size of the array, a fundamental question is the following. What
condition on pair $(C,\,A)$ guarantees the existence of an $L$ under
which the systems synchronize for all coupling configurations that
belong to the given set?  Clearly, the choice of the set shapes the
possible answers to this question. When the set is the collection of
all interconnections with coupling strength greater than some ({\em a
priori} known) positive threshold, the answer (which can be extracted,
for instance, from \cite{wu05}) is as simple as one can expect: that
$(C,\,A)$ is detectable, just as is the case with Luenberger
observer. Unlike Luenberger observer though, mere that $A-LC$ is
Hurwitz is not sufficient for synchronization. In attempt to make the
general picture closer to complete (see Theorem~\ref{thm:main}), in
this paper we focus on the {\em set of all interconnections with
connected graphs} (no assumption on the strength of coupling nor on
the symmetry or balancedness of the graph) and study new conditions on
$(C,\,A)$ which do (not) guarantee the existence of a synchronizing
feedback law with respect to that set. The main contribution of this
paper is in establishing the following results:
\begin{itemize}
\item[(b)] If $(C,\,A)$ is detectable with $A$ neutrally stable then there 
exists a synchronizing linear feedback law. (We also provide an
algorithm to explicitly compute one such feedback law.)
\item[(c)] If $A$ is critically unstable (e.g. double integrator) 
and $C$ is full column rank then there exists a synchronizing linear feedback 
law.
\item[(f)] If $(C,\,A)$ is detectable with $A$
critically unstable then, in general, there does not exist 
a synchronizing linear feedback law.
\end{itemize}
For the discrete-time version of (b) see \cite{aut7479}. 

In the remainder of the paper we first provide basic notation and
definitions. In Section~\ref{sec:problemstatement} we define {\em
synchronizability with respect to a set of interconnections} and
formalize the problem through that definition. Following the problem
statement we present the main theorem as a 8-item list of sufficient
and nonsufficient conditions for synchronizability. After stating the
main theorem (Theorem~\ref{thm:main}) we set out to prove what it
claims. The sufficiency statements of Theorem~\ref{thm:main} are
demonstrated in Section~\ref{sec:positive} where we also show that
coupled harmonic oscillators synchronize if the coupling is via a
connected graph (Corollary~\ref{cor:harmonic}.) To the best of our
knowledge, establishing synchronization of harmonic oscillators
without any symmetry, balancedness, or strong coupling assumption on
the interconnection is new. Section~\ref{sec:negative} includes the
proofs of nonsufficiency statements of Theorem~\ref{thm:main}, which
are given in order to give a measure on how tight (close to necessary)
the sufficient conditions are. We spend a few words on a dual problem
in Section~\ref{sec:dualproblem} after which we conclude.

\section{Notation and definitions}
Let $\Natural$ denote the set of nonnegative integers and $\Real_{\geq
0}$ set of nonnegative real numbers. Let $|\cdot|$ denote 2-norm. For
$\lambda\in\Complex$ let ${\rm Re}(\lambda)$ denote the real part of
$\lambda$.  Identity matrix in $\Real^{n\times n}$ is denoted by
$I_{n}$ and zero matrix in $\Real^{m\times n}$ by $0_{m\times n}$.
Conjugate transpose of a matrix $A$ is denoted by $A^{H}$. A matrix
$A\in\Complex^{n\times n}$ is {\em Hurwitz} if all of its eigenvalues
have strictly negative real parts. A matrix $S\in\Real^{n\times n}$ is
{\em skew-symmetric} if $S+S^{T}=0$. For $A\in\Real^{n\times n}$ and
$C\in\Real^{m\times n}$, pair $(C,\,A)$ is {\em detectable} (in the
continuous-time sense) if that $Ce^{At}x=0$ for some $x\in\Real^{n}$
and for all $t\geq 0$ implies $\lim_{t\to\infty}e^{At}x=0$. Matrix
$A\in\Real^{n\times n}$ is {\em neutrally stable} (in the
continuous-time sense) if it has no eigenvalue with positive real part
and the Jordan block corresponding to any eigenvalue on the imaginary
axis is of size one. Let $\one\in\Real^{p}$ denote the vector with all
entries equal to one.

{\em Kronecker product} of $A\in\Real^{m\times n}$ and $B\in\Real^{p\times q}$ is
\begin{eqnarray*}
A\otimes B:=
\left[\!\!
\begin{array}{ccc}
a_{11}B & \cdots & a_{1n}B\\
\vdots  & \ddots & \vdots\\
a_{m1}B & \cdots & a_{mn}B
\end{array}\!\!
\right]
\end{eqnarray*}
Kronecker product comes with the properties $(A\otimes
B)(C\otimes D)=(AC)\otimes(BD)$ (provided that products $AC$ and $BD$
are allowed) $A\otimes B+A\otimes C=A\otimes(B+C)$ (for $B$ and
$C$ that are of same size) and $(A\otimes B)^{T}=A^{T}\otimes B^{T}$. 

A ({\em directed}) {\em graph} is a pair $(\N,\,\setE)$ where $\N$ is a
nonempty finite set (of {\em nodes}) and $\setE$ is a finite collection
of ordered pairs ({\em edges}) $(n_{i},\,n_{j})$ with
$n_{i},\,n_{j}\in\setE$. A {\em path} from $n_{1}$ to $n_{\ell}$ is a
sequence of nodes $(n_{1},\,n_{2},\,\ldots,\,n_{\ell})$ such that
$(n_{i},\,n_{i+1})$ is an edge for $i\in\{1,\,2,\,\ldots,\,\ell-1\}$. A
graph is {\em connected} if it has a node to which there exists a path
from every other node.\footnote{Note that this definition of
connectedness for directed graphs is weaker than strong connectivity
and stronger than weak connectivity.} Fig.~\ref{fig:graph} illustrates
two graphs, where one is connected and the other is not. 

\begin{figure}[h]
\begin{center}
\includegraphics[scale=0.65]{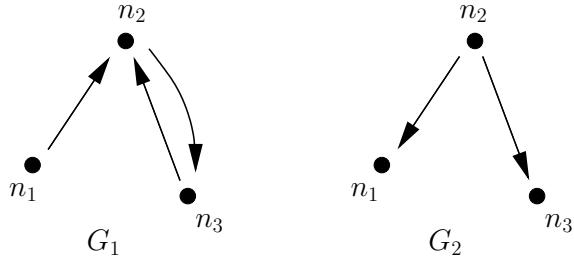}
\caption{Schematic representations of two graphs $G_{1}=(\N,\,\setE_{1})$ and 
$G_{2}=(\N,\,\setE_{2})$ where $\N=\{n_{1},\,n_{2},\,n_{3}\}$,
$\setE_{1}=\{(n_{1},\,n_{2}),\,(n_{2},\,n_{3}),\,(n_{3},\,n_{2})\}$ and
$\setE_{2}=\{(n_{2},\,n_{1}),\,(n_{2},\,n_{3})\}$. $G_{1}$ is connected,
$G_{2}$ is not.  }\label{fig:graph}
\end{center}
\end{figure}

A matrix $\Gamma:=[\gamma_{ij}]\in\Real^{p\times p}$ describes (is) an
{\em interconnection} if $\gamma_{ij}\geq 0$ for $i\neq j$ and
$\gamma_{ii}=-\sum_{j\neq i}\gamma_{ij}$. It immediately follows that
$\lambda=0$ is an eigenvalue with eigenvector $\one$
(i.e. $\Gamma\one=0$.) The graph of $\Gamma$ is the pair $(\N,\,\setE)$
where $\N =\{n_{1},\,n_{2},\,\ldots,\,n_{p}\}$ and
$(n_{i},\,n_{j})\in\setE$ iff $\gamma_{ij}>0$. Interconnection $\Gamma$
is said to be {\em connected} if its graph is connected.

For connected $\Gamma$, eigenvalue $\lambda=0$ is distinct and all the
other eigenvalues have real parts strictly negative.
When we write ${\rm Re}(\lambda_{2}(\Gamma))$ we mean the real part of
a nonzero eigenvalue of $\Gamma$ closest to the imaginary axis. Let
$r^{T}$ be the left eigenvector of eigenvalue $\lambda=0$
(i.e. $r^{T}\Gamma=0$) with $r^{T}\one=1$. Then $r^{T}$ is unique and
satisfies $\lim_{t\to\infty}e^{\Gamma{t}}={\one}r^{T}$.

Given maps $\xi_{i}:\Real_{\geq 0}\to\Real^{n}$ for
$i=1,\,2,\,\ldots,\,p$ and a map $\bar\xi:\Real_{\geq
0}\to\Real^{n}$, the elements of the set
$\{\xi_{i}(\cdot):i=1,\,2,\,\ldots,\,p\}$ are said to {\em synchronize
to} $\bar{\xi}(\cdot)$ if $|\xi_{i}(t)-\bar\xi(t)|\to 0$ as
$t\to\infty$ for all $i$. The elements of the set
$\{\xi_{i}(\cdot):i=1,\,2,\,\ldots,\,p\}$ are said to synchronize
if they synchronize to some $\bar{\xi}(\cdot)$. 

Let $\setP$ denote the set of all pairs $(C,\,A)$ where matrix
$C$ and square matrix $A$, both real, have the same number of
columns. We define the following subsets of $\setP$.
\begin{itemize}
\item $\A_{\rm H}$: set of all pairs $(C,\,A)$ with $A$ Hurwitz.
\item $\A_{\rm N}$: set of all pairs $(C,\,A)$ with $A$ neutrally stable. 
\item $\A_{\rm J}$: set of all pairs $(C,\,A)$ with $A$ having no 
eigenvalue with positive real part.
\item $\setO_{\rm F}$: set of all pairs $(C,\,A)$ with $C$ full column rank. 
\item $\setO_{\rm P}$: set of all detectable pairs. 
\end{itemize}
Few remarks are in order regarding the above definitions. Note that
$\A_{\rm H}\subset\A_{\rm N}\subset\A_{\rm J}$ and $\setO_{\rm
F}\subset\setO_{\rm P}$. Set $\A_{\rm J}$ allows $A$ matrices with
{\em Jordan} blocks of arbitrary size with eigenvalues on the
imaginary axis. (Hence the subscript J.) Therefore there are pairs
$(C,\,A)$ in $\A_{\rm J}$ with critically unstable $A$ matrices. For a
pair $(C,\,A)\in\setO_{\rm F}$ if $C$ is an output matrix of some
system of order $n$ then there exists matrix $L$ such that $LC=I_{n}$,
i.e. {\em full} state information is instantly available at the
output. (Hence the subscript F.) For an arbitrary
$(C,\,A)\in\setO_{\rm P}$, however, the state information is only
{\em partially} available at the output. (Hence the subscript P.) We also
need the following definitions.
\begin{itemize}
\item $\G_{\geq\delta}$: set of all connected interconnections with 
$|{\rm Re}(\lambda_{2}(\Gamma))|\geq \delta>0$. 
\item $\G_{>0}$: set of all connected interconnections.
\item $\G_{\geq 0}$: set of all interconnections. 
\end{itemize}
Note that $\G_{\geq\delta}\subset\G_{>0}\subset\G_{\geq 0}$.

\section{Problem statement and main theorem}\label{sec:problemstatement}  
For a given interconnection $\Gamma\in\Real^{p\times p}$, let an array of
$p$ linear systems be
\begin{subeqnarray}\label{eqn:system}
{\dot{x}}_{i} &=& Ax_{i}+u_{i} \label{eqn:a}\\
 y_{i} &=& Cx_{i}\label{eqn:b}\\
 z_{i} &=& \displaystyle \sum_{j\neq i}\gamma_{ij}(y_{j}-y_{i})\label{eqn:c}
\end{subeqnarray}
where $x_{i}\in\Real^{n}$ is the {\em state}, $u_{i}\in\Real^{n}$ is
the {\em input}, $y_{i}\in\Real^{m}$ is the {\em output}, and
$z_{i}\in\Real^{m}$ is the {\em coupling} of the $i$th system for
$i=1,\,2,\,\ldots,\,p$. Matrices $A$ and $C$ are of proper dimensions.
The solution of $i$th system at time $t\geq 0$ is denoted by
$x_{i}(t)$. 

\begin{definition}
Given $A\in\Real^{n\times n}$, $C\in\Real^{m\times n}$, and set of
interconnections $\setS$; pair $(C,\,A)$ is said to be {\em
synchronizable with respect to $\setS$} if there exists a linear feedback
law $L\in\Real^{n\times m}$ such that for each $\Gamma\in\setS$ solutions
$x_{i}(\cdot)$ of array~\eqref{eqn:system} with $u_{i}=Lz_{i}$
synchronize for all initial conditions.
\end{definition}
\begin{center}
{\em When is pair $(C,\,A)$ synchronizable with respect to
interconnection set $\G_{> 0}$?}
\end{center}

The above question is what we are mainly concerned with in this paper. Our
objective is to find sufficient conditions on $(C,\,A)$ for
synchronizability with respect to set of all connected
interconnections and examine their slackness. To this end, we adopt a
constructive approach, i.e. we compute a synchronizing feedback law
$L$ whenever possible. We also provide the explicit trajectory that
the solutions synchronize to (if they do synchronize.)  By adding our
findings to already known results on synchronization of coupled linear
systems we aim to reach a clearer picture of the problem with fewer
missing pieces, which is recapitulated in the below theorem.

\begin{theorem}\label{thm:main}
Let us be given $\delta>0$. We have the following.
\begin{itemize}
\item[(a)] All $(C,\,A)\in\A_{\rm H}$ are synchronizable with respect to $\G_{\geq 0}$. 
\item[(b)] All $(C,\,A)\in\A_{\rm N}\cap\setO_{\rm P}$ are synchronizable with respect to $\G_{>0}$.
\item[(c)] All $(C,\,A)\in\A_{\rm J}\cap\setO_{\rm F}$ are synchronizable with respect to $\G_{>0}$.
\item[(d)] All $(C,\,A)\in\setO_{\rm P}$ are synchronizable with respect to $\G_{\geq \delta}$.   
\item[(e)] Not all $(C,\,A)\in\A_{\rm N}\cap\setO_{\rm F}$ are synchronizable 
with respect to $\G_{\geq 0}$.
\item[(f)] Not all $(C,\,A)\in\A_{\rm J}\cap\setO_{\rm P}$ are synchronizable with respect to $\G_{>0}$.
\item[(g)] Not all $(C,\,A)\in\setO_{\rm F}$ are synchronizable with respect to $\G_{>0}$.   
\item[(h)] Not all $(C,\,A)\in\setP$ are synchronizable with respect to $\G_{\geq\delta}$.   
\end{itemize}
\end{theorem}  

As mentioned earlier, statement (d) has been known for a
while. General results pertaining to that case go as far back as
\cite{wu95}. Among the remaining statements, (a), (e), (g), and (h) are 
straightforward if not obvious. Still, for the sake of completeness, 
we establish each of the eight statements of Theorem~\ref{thm:main} by a series 
of lemmas. 

\section{Cases when synchronizing feedback exists}\label{sec:positive}
This section is dedicated to prove the existence statements of
Theorem~\ref{thm:main}, which are analysis results. However, what is
in here is more than mere showing the existence of some synchronizing
feedback law $L$ under different assumptions. For each case we
actually propose/compute an $L$. We also explicitly provide the
trajectory that systems will synchronize to under the proposed
feedback law. Therefore what will have been solved at the end is a
synthesis problem.
 
\subsection{System matrix Hurwitz} 
Let us be given some interconnection $\Gamma$. Consider
\eqref{eqn:system} under the feedback connection $u_{i}=Lz_{i}$ 
for some $L$.  When system matrices $A$ are Hurwitz, the trivial
choice $L=0$ decouples the systems and $x_{i}(t)\to 0$ as $t\to\infty$
for all initial conditions. Therefore synchronization accrues.  We
formalize this trivial observation to the following result and, consequently,
establish Theorem~\ref{thm:main}(a).

\begin{lemma}
Given $(C,\,A)\in\A_{\rm H}$, let $L:=0$. Then for all
$\Gamma\in\G_{\geq 0}$ solutions $x_{i}(\cdot)$ of
systems~\eqref{eqn:system} with $u_{i}=Lz_{i}$ synchronize to
$\bar{x}(t)\equiv 0$.
\end{lemma}

\subsection{System matrix neutrally stable} 
The second statement of Theorem~\ref{thm:main} is concerned with
neutrally stable system matrices. Here we choose to first study a
subcase (Proposition~\ref{prop:lyap}) where the system matrices are
skew-symmetric and then generalize what we obtain. For some
interconnection $\Gamma=[\gamma_{ij}]$ consider the following coupled
systems
\begin{eqnarray}\label{eqn:skew}
\dot\xi_{i}=S\xi_{i}+H^{T}H\sum_{j=1}^{p}\gamma_{ij}(\xi_{j}-\xi_{i})\, ,\quad i=1,\,2,\,\ldots,\,p
\end{eqnarray}
where $\xi_{i}\in\Real^{n}$ is the state of the $i$th system,
$S\in\Real^{n\times n}$, and $H\in\Real^{m\times n}$. We make the
following assumptions on systems~\eqref{eqn:skew} which will
henceforth hold.

\noindent
{\bf (A1)}\ $S$ is skew-symmetric.\\
{\bf (A2)}\ $(H,\,S)$ is observable.\\
{\bf (A3)}\ $\Gamma$ is connected.

\begin{proposition}\label{prop:lyap}
Consider systems~\eqref{eqn:skew}. Solutions
$\xi_{i}(\cdot)$ synchronize to
\begin{eqnarray*}
\bar{\xi}(t):=({r^{T}\otimes e^{St}})
\left[\!\!
\begin{array}{c}
\xi_{1}(0)\\
\vdots\\
\xi_{p}(0)
\end{array}
\!\!
\right]
\end{eqnarray*}
where $r\in\Real^{p}$ is such
that $r^{T}\Gamma=0$ and $r^{T}\one=1$.
\end{proposition}

\begin{proof}
Consider matrix $\Gamma-\one{r}^{T}$. Observe that
$(\Gamma-\one{r}^{T})^{k}=\Gamma^{k}+(-1)^{k}\one{r^{T}}$ for $k\in\Natural$. For
$t\in\Real$ therefore we can write
\begin{eqnarray*}
e^{(\Gamma-1r^{T})t}
&=&I_{p}+t(\Gamma-1r^{T})+\frac{t^{2}}{2}(\Gamma-1r^{T})^{2}+\ldots\\
&=&\left(I_{p}+t\Gamma+\frac{t^{2}}{2}\Gamma^{2}+\ldots\right)
-\left(t\one{r^{T}}-\frac{t^{2}}{2}\one{r}^{T}+\ldots\right)\\
&=&e^{\Gamma{t}}-(1-e^{-t})\one{r^{T}}\,.
\end{eqnarray*} 
Consequently $\lim_{t\to\infty}e^{(\Gamma-\one{r^{T}})t}=0$. We deduce
therefore that $\Gamma-\one{r^{T}}$ is Hurwitz. Since
$\Gamma-\one{r^{T}}$ is Hurwitz, there exist symmetric positive
definite matrices $P,\,Q\in\Real^{p\times p}$ such that
\begin{eqnarray}\label{eqn:prepost}
-Q=(\Gamma-\one{r^{T}})^{T}P+P(\Gamma-\one{r^{T}})\,.
\end{eqnarray}
Define positive semidefinite matrices
$\widehat{P}:=(I_{p}-\one{r^{T}})^{T}P(I_{p}-\one{r^{T}})$ and
$\widehat{Q}:=(I_{p}-\one{r^{T}})^{T}Q(I_{p}-\one{r^{T}})$. Now pre-
and post-multiply equation \eqref{eqn:prepost} by
$(I_{p}-\one{r^{T}})^{T}$ and $(I_{p}-\one{r^{T}})$, respectively. We
obtain
\begin{eqnarray*}
-\widehat{Q}
&=&(I_{p}-\one{r^{T}})^{T}(\Gamma-\one{r^{T}})^{T}P(I_{p}-\one{r^{T}})\nonumber\\
&&\qquad+(I_{p}-\one{r^{T}})^{T}P(\Gamma-\one{r^{T}})(I_{p}-\one{r^{T}})\nonumber\\
&=&\Gamma^{T}P(I_{p}-\one{r^{T}})+(I_{p}-\one{r^{T}})^{T}P\Gamma\nonumber\\
&=&\Gamma^{T}(I_{p}-\one{r^{T}})^{T}P(I_{p}-\one{r^{T}})
+(I_{p}-\one{r^{T}})^{T}P(I_{p}-\one{r^{T}})\Gamma\nonumber\\
&=&\Gamma^{T}\widehat{P}+\widehat{P}\Gamma\,.
\end{eqnarray*}
We now stack the individual system states to obtain
$\ex:=[\xi_{1}^{T}\ \xi_{2}^{T}\ \cdots\ \xi_{p}^T]^{T}$. We can then cast
\eqref{eqn:skew} into
\begin{eqnarray}\label{eqn:wrt}
\dot{\ex}=(I_{p}\otimes{S}+\Gamma\otimes{H^{T}H})\ex\,.
\end{eqnarray} 
Define $V:\Real^{pn}\to\Real_{\geq 0}$ as $V(\ex):=\ex^{T}(\widehat{P}\otimes I_{n})\ex$. 
Differentiating $V(\ex(t))$ with respect to time we obtain
\begin{eqnarray}\label{eqn:vdot}
\dot{V}(\ex)
&=&\ex^{T}(I_{p}\otimes{S^{T}}+\Gamma^{T}\otimes{H^{T}H})(\widehat{P}\otimes I_{n})\ex\nonumber\\
&&\qquad+\ex^{T}(\widehat{P}\otimes I_{n})(I_{p}\otimes{S}+\Gamma\otimes{H^{T}H})\ex\nonumber\\
&=&\ex^{T}(\widehat{P}\otimes({S^{T}}+S)
+(\Gamma^{T}\widehat{P}+\widehat{P}\Gamma)\otimes{H^{T}H})\ex\nonumber\\
&=&-\ex^{T}(\widehat{Q}\otimes{H^{T}H})\ex\,.
\end{eqnarray}
Thence $\dot{V}(\ex)\leq 0$ for both $\widehat{Q}$ and $H^{T}H$ (and
consequently their Kronecker product) are positive semidefinite.

Given some $\zeta\in\Real^{pn}$, let $\X\subset\Real^{pn}$ be the
closure of the set of all points $\eta$ such that
$\eta=(\one{r^{T}}\otimes e^{St})\zeta$ for some $t\geq 0$. Set $\X$
is compact for it is closed by definition and bounded due to that
$\zeta$ is fixed and $S$ is a neutrally-stable matrix. Having defined
$\X$, we now define 
\begin{eqnarray*}
\Omega:=\{\eta\in\Real^{pn}:(\one{r^{T}}\otimes
I_{n})\eta\in\X\,,V(\eta)\leq V(\zeta)\}\,.  
\end{eqnarray*}
Let us show that $\Omega$ is forward invariant. Observe that
\begin{eqnarray*}
\frac{d}{dt}\left((\one{r^{T}}\otimes I_{n})\ex(t)\right)
&=&(\one{r^{T}}\otimes I_{n})(I_{p}\otimes{S}+\Gamma\otimes{H^{T}H})\ex(t)\\
&=&(\one{r^{T}}\otimes S + \one{r^{T}}\Gamma\otimes{H^{T}H})\ex(t)\\
&=&(\one{r^{T}}\otimes S)\ex(t)\\
&=&(I_{p}\otimes{S})(\one{r^{T}}\otimes I_{n})\ex(t)\,.
\end{eqnarray*} 
We therefore have 
\begin{eqnarray}\label{eqn:aksuleyman}
(\one{r^{T}}\otimes I_{n})\ex(t)=(\one{r^{T}}\otimes e^{St})\ex(0)
\end{eqnarray}
which in turn implies that if $(\one{r^{T}}\otimes I_{n})\ex(0)\in\X$
then $(\one{r^{T}}\otimes I_{n})\ex(t)\in\X$ for all $t\geq
0$. Likewise, if $V(\ex(0))\leq V(\zeta)$ then $V(\ex(t))\leq
V(\zeta)$ for all $t\geq 0$ thanks to \eqref{eqn:vdot}. As a result,
if $\ex(0)\in\Omega$ then $\ex(t)\in\Omega$ for all $t\geq 0$, that
is, $\Omega$ is forward invariant with respect to \eqref{eqn:wrt}.

Set $\Omega$ is closed by construction. To show that it is compact
therefore all we need to do is to establish its boundedness. Let 
\begin{eqnarray*}
a:=\sup_{V(\eta)\leq V(\zeta)}|\eta-(\one{r^{T}}\otimes I_{n})\eta|\,.
\end{eqnarray*}
If we go back to the definition of $V$ we immediately see that $a<\infty$. Now let
\begin{eqnarray*}
b:=\sup_{\omega\in\X}|\omega|\,.
\end{eqnarray*}
Since $\X$ is bounded, $b<\infty$ as well. Now, given any $\eta\in\Omega$ we have
$|\eta-(\one{r^{T}}\otimes I_{n})\eta|\leq a$. Hence we can write
\begin{eqnarray*}
|\eta|
&\leq& a+|(\one{r^{T}}\otimes I_{n})\eta|\\
&\leq& a+\sup_{\omega\in\X}|\omega|\\
&=& a+b\,.
\end{eqnarray*}
Therefore $\Omega$ is bounded. Having shown that $\Omega$ is forward
invariant and compact, we can now invoke LaSalle's invariance
principle \cite[Thm.~3.4]{khalil96} and claim that any solution
starting in $\Omega$ approaches to the largest invariant set
$\W\subset\{\eta\in\Omega:\dot{V}(\eta)=0\}$.

Let now $\eta(\cdot)$ be a solution of \eqref{eqn:wrt} such that $\eta(t)\in\W$ for all $t\geq 0$. 
Given some $\tau\geq 0$, since $\dot{V}(\eta(\tau))=0$, we can write 
\begin{eqnarray*}
0
&=&\eta(\tau)^{T}(\widehat{Q}\otimes H^{T}H)\eta(\tau)\\
&=&\eta(\tau)^{T}((I_{p}-\one{r^{T}})^{T}Q(I_{p}-\one{r^{T}})\otimes H^{T}H)\eta(\tau)
\end{eqnarray*} 
which implies, since $Q$ is positive definite, that either
$((I_{p}-\one{r^{T}})\otimes I_{n})\eta(\tau)=0$ or $(I_{p}\otimes
H)\eta(\tau)=0$. Suppose now that 
\begin{eqnarray}\label{eqn:star}
((I_{p}-\one{r^{T}})\otimes I_{n})\eta(\tau)\neq 0\,. 
\end{eqnarray}
Continuity of $\eta(\cdot)$ implies that there exists $\delta>0$ such
that $((I_{p}-\one{r^{T}})\otimes I_{n})\eta(t)\neq 0$ for
$t\in[\tau,\,\tau+\delta]$. Therefore we must have $(I_{p}\otimes
H)\eta(t)=0$ for $t\in[\tau,\,\tau+\delta]$. However, observability of
pair $(H,\,S)$ stipulates that $\eta(t)=0$ for
$t\in[\tau,\,\tau+\delta]$ which contradicts \eqref{eqn:star}. We
then deduce $((I_{p}-\one{r^{T}})\otimes I_{n})\eta(t)=0$ for all
$t\geq 0$.  Therefore
$\W\subset\{\omega\in\Omega:\omega=(\one{r^{T}}\otimes
I_{n})\omega\}=\X$.

Let us now be given any solution $\ex(\cdot)$ of \eqref{eqn:wrt}.
Since $\zeta$ that we used to construct $\Omega$ was arbitrary, without
loss of generality, we can take $\ex(0)=\zeta$. That $\ex(0)\in\Omega$
implies that $\ex(t)$ approaches $\X$ as $t\to\infty$. Therefore we
are allowed to write
\begin{eqnarray*}
0
&=&\lim_{t\to \infty}\left(\ex(t)-(\one{r^{T}}\otimes I_{n})\ex(t)\right)\\
&=&\lim_{t\to \infty}\left(\ex(t)-(\one{r^{T}}\otimes e^{St})\ex(0)\right)
\end{eqnarray*}
where we used \eqref{eqn:aksuleyman}.
\end{proof}

The following result, which establishes synchronization of coupled
harmonic oscillators, comes as a byproduct of
Proposition~\ref{prop:lyap}.

\begin{corollary}\label{cor:harmonic}
Consider $p$ coupled harmonic oscillators (in $\Real^{2}$) described by
\begin{eqnarray*}
\dot{x}_{i}&=& y_{i}\\
\dot{y}_{i}&=& -x_{i}+\sum_{j=1}^{p}\gamma_{ij}(y_{j}-y_{i})\,.
\end{eqnarray*}
Oscillators synchronize for all connected interconnections $[\gamma_{ij}]$.
\end{corollary}
We will use the next fact in Algorithm~\ref{alg:L} to construct a
feedback law for synchronization.

\begin{fact}\label{fact:one}
Let $F\in\Real^{n\times n}$ be a neutrally-stable matrix with all its eigenvalues 
residing on the imaginary axis. Then 
\begin{eqnarray}\label{eqn:limit}
P:=\lim_{t\to\infty}\ t^{-1}\int_{0}^{t}e^{F^{T}\tau}e^{F\tau}d\tau
\end{eqnarray}
is well-defined and symmetric positive definite. It also satisfies $PF+F^{T}P=0$.
\end{fact}

\begin{proof}
Matrix $F$ is similar to a skew-symmetric matrix. Therefore $e^{Ft}$
is (almost) periodic \cite{vanvleck64}. Periodicity directly yields
that limit in \eqref{eqn:limit} exists, that is, $P$ is well-defined.
Similarity to a skew-symmetric matrix also brings that $\inf_{t\in\Real}|e^{Ft}|>0$
and $\sup_{t\in\Real}|e^{Ft}|<\infty$. Same goes for
$F^{T}$. Therefore there exist scalars $a,\,b>0$ such that $aI_{n}\leq
e^{F^{T}t}e^{Ft}\leq bI_{n}$ for all $t\in\Real$. We can then write
\begin{eqnarray*}
aI_{n}\leq t^{-1}\int_{0}^{t}e^{F^{T}\tau}e^{F\tau}d\tau\leq bI_{n}
\end{eqnarray*}
for all $t\geq 0$. Therefore $P$ is positive definite. Symmetricity of $P$ comes 
by construction. Finally, observe that
\begin{eqnarray*}
|PF+F^{T}P|
&=&\lim_{t\to\infty}\ t^{-1}\left|\int_{0}^{t}\left(e^{F^{T}\tau}e^{F\tau}F
+F^{T}e^{F^{T}\tau}e^{F\tau}\right)d\tau\right|\\
&=&\lim_{t\to\infty}\ t^{-1}\left|\int_{0}^{t}d\left(e^{F^{T}\tau}e^{F\tau}\right)\right|\\
&\leq&\lim_{t\to\infty}\ t^{-1}\left(\left|e^{F^{T}t}e^{Ft}\right|
+\left|e^{F^{T}0}e^{F0}\right|\right)\\
&\leq&\lim_{t\to\infty}\ t^{-1}(b+1)\\
&=&0
\end{eqnarray*}
whence the result follows.
\end{proof}

\begin{algorithm}\label{alg:L}
Given $A\in\Real^{n\times n}$ that is neutrally stable and
$C\in\Real^{m\times n}$, we obtain $L\in\Real^{n\times m}$ as
follows. Let $n_{1}\leq n$ be the number of eigenvalues of $A$ that
reside on the imaginary axis. Let $n_{2}:=n-n_{1}$.  If $n_{1}=0$,
then let $L:=0$; else construct $L$ according to the following steps.

\noindent
{\em Step 1:} Choose $U\in\Real^{n\times n_{1}}$ and $W\in\Real^{n\times n_{2}}$ satisfying
\begin{eqnarray*}
[U\ W]^{-1}A[U\ W]=
\left[\!\!
\begin{array}{cc}
F & 0\\
0 & G
\end{array}\!\!
\right]
\end{eqnarray*} 
where all the eigenvalues of $F\in\Real^{n_{1}\times n_{1}}$ have zero real parts. 

\noindent
{\em Step 2:} Obtain $P\in\Real^{n_{1}\times n_{1}}$ from $F$ by \eqref{eqn:limit}. 

\noindent
{\em Step 3:} Finally let $L:=UP^{-1}(CU)^{T}$.
\end{algorithm} 

Below result establishes Theorem~\ref{thm:main}(b).

\begin{lemma}\label{lem:b}
Given $(C,\,A)\in\A_{\rm N}\cap\setO_{\rm P}$, let $L$ be constructed
according to Algorithm~\ref{alg:L}. Then for all $\Gamma\in\G_{> 0}$
solutions $x_{i}(\cdot)$ of systems~\eqref{eqn:system} with
$u_{i}=Lz_{i}$ synchronize to
\begin{eqnarray*}
\bar{x}(t):=({r^{T}\otimes e^{At}})
\left[\!\!
\begin{array}{c}
x_{1}(0)\\
\vdots\\
x_{p}(0)
\end{array}
\!\!
\right]
\end{eqnarray*}
where $r\in\Real^{p}$ is such that $r^{T}\Gamma=0$ and $r^{T}\one=1$.
\end{lemma}

\begin{proof}
Let the variables that are not introduced here be defined as in
Algorithm~\ref{alg:L}.  Let $H:=CUP^{-1/2}$ and
$S:=P^{1/2}FP^{-1/2}$. Then $(H,\,S)$ is observable for $(C,\,A)$ is
detectable. Also, note that $S$ is skew-symmetric due to $PF+F^{T}P=0$.

We let $U^{\dagger}\in\Real^{n_{1}\times n}$ and $W^{\dagger}\in\Real^{n_{2}\times n}$ be such that
\begin{eqnarray*}
\left[\!\!
\begin{array}{c}
U^{\dagger}\\
W^{\dagger}
\end{array}\!\!
\right]=[U\ W]^{-1}\,.
\end{eqnarray*} 
Note then that $U^{\dagger}U=I_{n_{1}}$, $W^{\dagger}W=I_{n_{2}}$, $U^{\dagger}W=0$, and $W^{\dagger}U=0$.
Since $u_{i}=Lz_{i}$, from \eqref{eqn:system} we obtain
\begin{eqnarray}\label{eqn:piece1}
\dot{x}_{i}=Ax_{i}+LC\sum_{j=1}^{p}\gamma_{ij}(x_{j}-x_{i})
\end{eqnarray}
Let now $\xi_{i}\in\Real^{n_{1}}$ and $\eta_{i}\in\Real^{n_{2}}$ be 
\begin{eqnarray}\label{eqn:piece2}
\left[\!\!
\begin{array}{c}
\xi_{i}\\
\eta_{i}\\
\end{array}\!\!
\right]:=
\left[\!\!
\begin{array}{cc}
P^{1/2}&0\\
0& I_{n_{2}}\\
\end{array}\!\!
\right]
\left[\!\!
\begin{array}{c}
U^{\dagger}\\
W^{\dagger}
\end{array}\!\!
\right]x_{i}
\end{eqnarray}
Combining \eqref{eqn:piece1} and \eqref{eqn:piece2} we can write
\begin{eqnarray}
\dot{\xi}_{i}&=&S\xi_{i}
+H^{T}H\sum_{j=1}^{p}\gamma_{ij}(\xi_{j}-\xi_{i})
+H^{T}CW\sum_{j=1}^{p}\gamma_{ij}(\eta_{j}-\eta_{i})\label{eqn:badem1}\\
\dot{\eta}_{i}&=&G\eta_{i}\,.\label{eqn:badem2}
\end{eqnarray}
Let $\Gamma\in\Real^{p\times p}$ be connected and $r\in\Real^{p}$ be such that
$r^{T}\Gamma=0$. Then define $\omega_{i}:\Real_{\geq
0}\to\Real^{n_{1}}$ as $\omega_{i}(t):=e^{-St}\xi_{i}(t)$ for
$i=1,\,2,\,\ldots,\,p$. Let $\dabilyu:=[\omega_{1}^{T}\ \omega_{2}^{T}\ \ldots\
\omega_{p}^{T}]^{T}$ and $\vi:=[\eta_{1}^{T}\ \eta_{2}^{T}\ \ldots\
\eta_{p}^{T}]^{T}$. Starting from \eqref{eqn:badem1} and
\eqref{eqn:badem2} we can write
\begin{eqnarray*}
\dot{\dabilyu}(t)&=&(\Gamma\otimes e^{-St}H^{T}He^{St})\dabilyu(t)
+(\Gamma\otimes e^{-St}H^{T}CWe^{Gt})\vi(0)\,.
\end{eqnarray*}
Thence
\begin{eqnarray}\label{eqn:integral}
\dabilyu(t)=\Phi(t,\,0)\dabilyu(0)
+\left[\int_{0}^{t}\Phi(t,\,\tau)(\Gamma\otimes e^{-S\tau}H^{T}CWe^{G\tau})d\tau\right]\vi(0)
\end{eqnarray}
where 
\begin{eqnarray*}
\Phi(t,\,\tau):={\rm exp}\left(
\int_{\tau}^{t}\left(\Gamma\otimes e^{-S\alpha}H^{T}He^{S\alpha}\right)d\alpha\right)
\end{eqnarray*}
is the state transition matrix \cite{antsaklis97}. From
Proposition~\ref{prop:lyap} we can deduce that $\Phi(t,\,\tau)$ is
uniformly bounded for all $t$ and $\tau$. Also, for any fixed $\tau$
we have $\lim_{t\to\infty}\Phi(t,\,\tau)=\one{r^{T}}\otimes
I_{n_{1}}$. Moreover, $e^{St}$ is uniformly bounded for all $t$, and
$e^{Gt}$ decays exponentially as $t\to\infty$ for $G$ is
Hurwitz. Therefore we can write
\begin{eqnarray*}
\lefteqn{\lim_{t\to\infty}\int_{0}^{t}\Phi(t,\,\tau)
(\Gamma\otimes e^{-S\tau}H^{T}CWe^{G\tau})d\tau}\\
&&\qquad\qquad\qquad=\int_{0}^{\infty}\left(\lim_{t\to\infty}\Phi(t,\,\tau)\right)
(\Gamma\otimes e^{-S\tau}H^{T}CWe^{G\tau})d\tau\\
&&\qquad\qquad\qquad=\int_{0}^{\infty}(\one{r^{T}}\otimes I_{n_{1}})(\Gamma\otimes e^{-S\tau}H^{T}CWe^{G\tau})d\tau\\
&&\qquad\qquad\qquad=0\,.
\end{eqnarray*}
Then, by \eqref{eqn:integral}, we can write
\begin{eqnarray*}
\lim_{t\to\infty}\dabilyu(t)=(\one{r^{T}}\otimes I_{n_{1}})\dabilyu(0)\,.
\end{eqnarray*}
Therefore solutions $\xi_{i}(\cdot)$ synchronize to $(r^{T}\otimes
e^{St})\dabilyu(0)$. Moreover, $\lim_{t\to\infty}\vi(t)=0$ for $G$ is
Hurwitz. Hence we can say that solutions $\eta_{i}(\cdot)$ synchronize to
$(r^{T}\otimes e^{Gt})\vi(0)$. As a result, solutions $x_{i}(\cdot)$ synchronize to
\begin{eqnarray*}
\left(r^{T}\otimes 
\left[UP^{-1/2}\ \ W\right]
\left[\!\!
\begin{array}{cc}
e^{St}&0\\0&e^{Gt}
\end{array}\!\!
\right]
\left[\!\!\begin{array}{c}P^{1/2}U^{\dagger}\\ W^{\dagger}\end{array}\!\!\right]
\right)
\left[\!\!
\begin{array}{c}
x_{1}(0)\\
\vdots\\
x_{p}(0)
\end{array}\!\!
\right]&&\\
=(r^{T}\otimes e^{At})
\left[\!\!
\begin{array}{c}
x_{1}(0)\\
\vdots\\
x_{p}(0)
\end{array}\!\!
\right]&&
\end{eqnarray*}
Hence the result.
\end{proof}

\subsection{System matrix critically unstable}
The third statement of Theorem~\ref{thm:main} is not surprising. The
naive reasoning, which we will later have no difficulty in
formalizing, is that when the output matrix $C$ is full column rank,
there exists an $L$ such that $LC$ equals the identity matrix. Then,
see \eqref{eqn:piece1}, there will be two {\em forces} driving the
array. One is due to system matrices $A$, the other due to coupling
between the systems. When $A$ matrices have no eigenvalues with
positive real part, even if the systems tend to move away from each
other, that tendency will not be of exponential nature. However, if
$\Gamma$ is connected, that apart-drifting behaviour will be dominated
by the exponential attraction introduced by the coupling between the
systems. The result will be synchronization of the solutions. Having
said that, now we can state our main reason for including this
existence statement in Theorem~\ref{thm:main}: to emphasize the sixth
statement of Theorem~\ref{thm:main}, which is a nonexistence
result. When availability of full-state information is replaced with
detectability, which in all the other cases we study causes no
problem, synchronizability is lost for critically unstable
systems. That is, Theorem~\ref{thm:main}(f) holds {\em despite}
Theorem~\ref{thm:main}(c), which we find counterintuitive. Below we
establish Theorem~\ref{thm:main}(c).

\begin{lemma}\label{lem:c}
Given $(C,\,A)\in\A_{\rm J}\cap\setO_{\rm F}$, let $L:=(C^{T}C)^{-1}C^{T}$. 
Then for all $\Gamma\in\G_{> 0}$
solutions $x_{i}(\cdot)$ of systems~\eqref{eqn:system} with
$u_{i}=Lz_{i}$ synchronize to
\begin{eqnarray*}
\bar{x}(t):=({r^{T}\otimes e^{At}})
\left[\!\!
\begin{array}{c}
x_{1}(0)\\
\vdots\\
x_{p}(0)
\end{array}
\!\!
\right]
\end{eqnarray*}
where $r\in\Real^{p}$ is such that $r^{T}\Gamma=0$ and $r^{T}\one=1$.
\end{lemma}

\begin{proof}
For $u_{i}=Lz_{i}$ and some connected $\Gamma\in\Real^{p\times p}$,
with $r\in\Real^{p}$ such that $r^{T}\Gamma=0$ and $r^{T}\one=1$, from
\eqref{eqn:system} one can write
\begin{eqnarray*}
\dot{x}_{i}
&=&Ax_{i}+LC\sum_{j=1}^{p}\gamma_{ij}(x_{j}-x_{i})\\
&=&Ax_{i}+\sum_{j=1}^{p}\gamma_{ij}(x_{j}-x_{i})
\end{eqnarray*}
since $LC=I_{n}$. Letting $\ex:=[x_{1}^{T}\ x_{2}^{T}\ \ldots\
x_{p}^{T}]^{T}$ we obtain
\begin{eqnarray*}
\dot{\ex}=(I_{p}\otimes A+\Gamma\otimes I_{n})\ex
\end{eqnarray*}
which implies 
\begin{eqnarray*}
\ex(t)=(e^{\Gamma t}\otimes e^{At})\ex(0)\,.
\end{eqnarray*}
Let us define 
\begin{eqnarray}
\tilde\ex(t)
&:=&\ex(t)-(\one r^{T}\otimes e^{At})\ex(0)\nonumber \\
&=&(e^{\Gamma t}\otimes e^{At})\ex(0)-(\one r^{T}\otimes e^{At})\ex(0)\nonumber \\
&=&((e^{\Gamma t}-\one r^{T})\otimes e^{At})\ex(0)\,. \label{eqn:sekersu}
\end{eqnarray}
From the proof of Theorem~\ref{lem:b} we know that $e^{\Gamma
t}-\one r^{T}=e^{(\Gamma-\one r^{T})t}+e^{-t}\one{r^{T}}$ and that
$\Gamma-\one r^{T}$ is Hurwitz. Also recall that $A$ does not have any
eigenvalues with positive real part. Therefore there exist
$M_{g},\,M_{a},\,\lambda>0$, and integer $k\leq n-1$ such that
$|e^{(\Gamma-\one{r^{T}})t}|\leq M_{g}e^{-\lambda t}$ and
$|e^{At}|\leq M_{a}t^{k}$. Then we can proceed from \eqref{eqn:sekersu} as
\begin{eqnarray*}
\tilde\ex(t)
&=&((e^{\Gamma t}-\one r^{T})\otimes e^{At})\ex(0)\\
&=&(e^{(\Gamma-\one r^{T})t}\otimes e^{At})\ex(0)+e^{-t}(\one r^{T}\otimes e^{At})\ex(0)
\end{eqnarray*}
whence
\begin{eqnarray*}
|\tilde\ex(t)|
&\leq& |e^{(\Gamma-\one r^{T})t}|\cdot|e^{At}|\cdot|\ex(0)|+e^{-t}|\one r^{T}|\cdot|e^{At}|\cdot|\ex(0)|\\
&\leq& M_{g}e^{-\lambda t}M_{a}t^{k}|\ex(0)|+e^{-t}M_{a}t^{k}|\one r^{T}|\cdot|\ex(0)|
\end{eqnarray*}
whence $\lim_{t\to\infty}|\tilde\ex(t)|=0$. Hence the result.
\end{proof} 

\subsection{System matrix arbitrary}
The fourth statement of Theorem~\ref{thm:main} is a special case of a
well-studied problem \cite{wu95,pecora98,belykh06,wu05,wu05b} which is
based on the idea that {\em identical systems synchronize under strong
enough coupling.} For coherence we regenerate a proof here. We first
borrow a well-known result from optimal control theory
\cite{sontag98}: Given a detectable pair $(C,\,A)$, where
$C\in\Real^{m\times n}$ and $A\in\Real^{n\times n}$, the following
algebraic Riccati equation
\begin{eqnarray}\label{eqn:riccati}
AP+PA^{T}+I_{n}-PC^{T}CP=0
\end{eqnarray}  
has a (unique) solution $P=P^{T}>0$. One can rewrite
\eqref{eqn:riccati} as
\begin{eqnarray*}
(A-PC^{T}C)P+P(A-PC^{T}C)^{T}+(I_{n}+PC^{T}CP)=0
\end{eqnarray*}
whence we infer that $A-PC^{T}C$ is Hurwitz. 

\begin{claim}\label{clm:lqr}
Let $C\in\Real^{m\times n}$ and $A\in\Real^{n\times n}$ satisfy
\eqref{eqn:riccati} for some symmetric positive definite $P$. Then for
all $\sigma\geq 1$ and $\omega\in\Real$ matrix
$A-(\sigma+j\omega)PC^{T}C$ is Hurwitz.
\end{claim}

\begin{proof}
Let $\varepsilon:=\sigma-1\geq 0$. Write 
\begin{eqnarray}\label{eqn:complexlyap}
\lefteqn{(A-(\sigma+j\omega)PC^{T}C)P+P(A-(\sigma+j\omega)PC^{T}C)^{H}}\nonumber\\
&&=(A-(\sigma+j\omega)PC^{T}C)P+P(A-(\sigma-j\omega)PC^{T}C)^{T}\nonumber\\
&&=(A-(1+\varepsilon)PC^{T}C)P+P(A-(1+\varepsilon)PC^{T}C)^{T}\nonumber\\
&&=(A-PC^{T}C)P+P(A-PC^{T}C)^{T}-2\varepsilon PC^{T}CP\nonumber\\
&&=-I_{n}-(1+2\varepsilon)PC^{T}CP\,.
\end{eqnarray}
Observe that
\eqref{eqn:complexlyap} is nothing but (complex) Lyapunov equation.
\end{proof}

Below result establishes Theorem~\ref{thm:main}(d).
\begin{lemma}
Given $\delta>0$ and $(C,\,A)\in \setO_{\rm P}$, let
$L:=\max\{1,\,\delta^{-1}\}PC^{T}$ where $P$ is the solution to
\eqref{eqn:riccati}. Then for all $\Gamma\in\G_{\geq \delta}$
solutions $x_{i}(\cdot)$ of systems~\eqref{eqn:system} with
$u_{i}=Lz_{i}$ synchronize to
\begin{eqnarray*}
\bar{x}(t):=({r^{T}\otimes e^{At}})
\left[\!\!
\begin{array}{c}
x_{1}(0)\\
\vdots\\
x_{p}(0)
\end{array}
\!\!
\right]
\end{eqnarray*}
where $r\in\Real^{p}$ is such that $r^{T}\Gamma=0$ and $r^{T}\one=1$.
\end{lemma}

\begin{proof}
Given $\Gamma\in\G_{\geq \delta}$, from \eqref{eqn:system} we can write
\begin{eqnarray}\label{eqn:couplexi}
\dot{x}_{i}=Ax_{i}+LC\sum_{j=1}^{p}\gamma_{ij}(x_{j}-x_{i})\,.
\end{eqnarray}
Stack individual system states as
$\ex:=[x_{1}^{T}\ x_{2}^{T}\ \ldots\ x_{p}^{T}]^{T}$. Then from 
\eqref{eqn:couplexi} we can write
\begin{eqnarray}\label{eqn:couplex}
\dot{\ex}=(I_{p}\otimes A + \Gamma\otimes LC)\ex\,.
\end{eqnarray}
Now let $Y\in\Complex^{p\times (p-1)}$, $W\in\Complex^{(p-1)\times p}$,
$V\in\Complex^{p\times p}$, and upper triangular
$\Delta\in\Complex^{(p-1)\times(p-1)}$ be such that
\begin{eqnarray*}
V=\left[\one\ Y\right]\ ,\quad 
V^{-1}=
\left[\!\!
\begin{array}{c}
r^{T}\\
W
\end{array}\!\!
\right]
\end{eqnarray*} 
and 
\begin{eqnarray*}
V^{-1}\Gamma V=
\left[\!\!
\begin{array}{cc}
0&0\ \cdots\ 0\\
\begin{array}{c}
0\\
\vdots\\
0
\end{array}
& \Delta
\end{array}\!\!
\right]
\end{eqnarray*}
Note that the diagonal entries of $\Delta$ are nothing but the nonzero
eigenvalues of $\Gamma$ which we know have real parts no greater than
$-\delta$. Engage the change of variables $\vi:=(V^{-1}\otimes
I_{n})\ex$ and modify
\eqref{eqn:couplex} first into
\begin{eqnarray*}
\dot{\vi}=(I_{p}\otimes A+V^{-1}\Gamma V\otimes LC)\vi
\end{eqnarray*}
and then into
\begin{eqnarray}\label{eqn:vi}
\dot{\vi}=
\left[\!\!
\begin{array}{cc}
A&0_{n\times(p-1)n}\\
0_{(p-1)n\times n}& I_{p-1}\otimes A+\Delta\otimes LC
\end{array}\!\!
\right]\vi
\end{eqnarray}
Observe that $I_{p-1}\otimes A+\Delta\otimes LC$ is upper block
triangular with (block) diagonal entries of the form $A+\lambda_{i}
\max\{1,\,\delta^{-1}\}PC^{T}C$ for $i=2,\,3,\,\ldots,\,p$ with ${\rm Re}(\lambda_{i})\leq
-\delta$. Claim~\ref{clm:lqr} implies therefore that $I_{p-1}\otimes
A+\Delta\otimes LC$ is Hurwitz. Thus \eqref{eqn:vi} implies
\begin{eqnarray*}
\lim_{t\to\infty}\left|\vi(t)-\left[\!\!
\begin{array}{cc}
e^{At}&0_{n\times(p-1)n}\\
0_{(p-1)n\times n}& 0_{(p-1)n\times(p-1)n}
\end{array}\!\!
\right]\vi(0)\right|= 0
\end{eqnarray*} 
which yields
\begin{eqnarray*}
\lim_{t\to\infty}\left|\ex(t)-(\one r^{T}\otimes e^{At})\ex(0)\right|= 0\,.
\end{eqnarray*} 
Hence the result.
\end{proof}

\section{Cases when no synchronizing feedback exists}\label{sec:negative}
In this section the nonsufficiency statements of the main theorem are
demonstrated. Below we establish Theorem~\ref{thm:main}(e).

\begin{lemma}
There exists $(C,\,A)\in\A_{\rm N}\cap\setO_{\rm F}$ that is not
synchronizable with respect to $\G_{\geq 0}$.
\end{lemma}

\begin{proof}
Suppose not. Then for $C=1$ and $A=0$ (note that $(C,\,A)\in\A_{\rm
N}\cap\setO_{\rm F}$) there exists $L\in\Real$ such that for all
$\Gamma\in\G_{\geq 0}$ solutions of systems~\eqref{eqn:system} with
$u_{i}=Lz_{i}$ synchronize for all initial conditions. However for
$\Gamma:=0_{2\times 2}\in\G_{\geq 0}$ we have $\dot{x}_{i}=0$, for
$i=1,\,2$, regardless of $L$. Therefore solutions $x_{i}(\cdot)$ do
not synchronize for $x_{1}(0)\neq x_{2}(0)$. The result follows by
contradiction.
\end{proof}

Next result establishes Theorem~\ref{thm:main}(f). It emphasizes the
nonsufficiency of partial-state coupling for synchronizability of
critically unstable systems with respect to set all connected
interconnections. As shown by Lemma~\ref{lem:c} this lost
synchronizability can be restored under full-state coupling.
\begin{lemma}
There exists $(C,\,A)\in\A_{\rm J}\cap\setO_{\rm P}$ that is not
synchronizable with respect to $\G_{> 0}$.
\end{lemma}

\begin{proof}
Suppose not. Let 
\begin{eqnarray*}
C:=[1\ \ 0]\ ,\qquad
A:=\left[\!\!
\begin{array}{rr}
0&1\\0&0
\end{array}\!\!
\right]
\end{eqnarray*}
Note that $(C,\,A)\in\A_{\rm J}\cap\setO_{\rm P}$. Then there exists 
$L\in\Real^{2\times 1}$ 
\begin{eqnarray*}
L:=\left[\!\!
\begin{array}{c}
\ell\\
\rho
\end{array}
\!\!
\right]
\end{eqnarray*}
such that for all $\Gamma\in\G_{> 0}$
solutions of systems~\eqref{eqn:system} with $u_{i}=Lz_{i}$
synchronize for all initial conditions. That is equivalent, by
\cite{pecora98}, to that $A+\lambda LC$ is Hurwitz for all 
$\lambda\in\{\gamma\in\Complex:0\neq\gamma\ 
\mbox{is an eigenvalue of some}\ \Gamma\in\G_{>0}\}=:\setL$. 
Note that 
\begin{eqnarray*}
\Gamma\in\G_{>0} \implies r\Gamma\in\G_{>0}\quad \forall r>0 
\end{eqnarray*}
whence 
\begin{eqnarray}\label{eqn:implies1}
\lambda\in\setL \implies r\lambda\in\setL\quad \forall r>0\,.  
\end{eqnarray}
Observe that any $\Gamma\in\G_{>0}\cap\Real^{2\times 2}$ has a
negative (real) eigenvalue. By \eqref{eqn:implies1} therefore $\setL$ includes
negative real line. Then a simple calculation shows that both
$\ell$ and $\rho$ need be strictly positive which lets us, thanks to
\eqref{eqn:implies1}, take $\ell=2$ without loss of
generality. Thence $A+\lambda LC$ reduces to the following matrix 
\begin{eqnarray}\label{eqn:matrix}
\left[\!\!
\begin{array}{cc}
2\lambda & 1\\
\lambda\rho & 0
\end{array}\!\!
\right]
\end{eqnarray}
which must be Hurwitz for all $\lambda\in\setL$. The roots of the
characteristic equation of the matrix in \eqref{eqn:matrix} is
\begin{eqnarray*}
s_{1,2}=\lambda\pm\sqrt{\lambda^{2}+\lambda\rho}\,.
\end{eqnarray*}
Then for $\theta\in(\pi/2,\,3\pi/2)$ we can write
\begin{eqnarray*}
{\rm Re}(s_{1,2})\big|_{\lambda=e^{j\theta}}
&=&\cos\theta\pm\frac{1}{\sqrt{2}}\sqrt{\cos2\theta+\rho\cos\theta
+\sqrt{(\cos2\theta+\rho\cos\theta)^{2}+(\sin2\theta+\rho\sin\theta)^{2}}}\\
&=&\cos\theta\pm\frac{1}{\sqrt{2}}\sqrt{\cos2\theta+\rho\cos\theta
+\sqrt{1+\rho^{2}+2\rho\cos\theta}}\,.
\end{eqnarray*}
Observe that we can choose $\bar\theta\in(\pi/2,\,\pi]$ and $m>0$ such that 
the following inequalities are satisfied
\begin{eqnarray*}
1+\frac{9\rho^{2}}{m^{2}} &\leq& \sqrt{1+\rho^{2}+2\rho\cos\theta}\\
-\rho\cos\theta &\leq& \frac{\rho^{2}}{m^{2}}\\
-\cos\theta &\leq& \frac{\rho}{m}
\end{eqnarray*}
for all $\theta\in(\pi/2,\,\bar\theta]$. Then for all $\theta\in(\pi/2,\,\bar\theta]$ 
one of the roots, say $s_{1}$, satisfies
\begin{eqnarray*}
{\rm Re}(s_{1})&=& \cos\theta+\frac{1}{\sqrt{2}}\sqrt{\cos2\theta+\rho\cos\theta
+\sqrt{1+\rho^{2}+2\rho\cos\theta}}\\
&\geq& \cos\theta+\frac{1}{\sqrt{2}}\sqrt{\cos2\theta-\frac{\rho^{2}}{m^{2}}
+1+\frac{9\rho^{2}}{m^{2}}}\\
&=& \cos\theta+\frac{1}{\sqrt{2}}\sqrt{2\cos^{2}\theta+\frac{8\rho^{2}}{m^{2}}}\\
&\geq& \cos\theta+\frac{2\rho}{m}\\
&\geq& \frac{\rho}{m}\\
&>& 0
\end{eqnarray*}
which implies, since no eigenvalue of matrix \eqref{eqn:matrix} can
have positive real part, $e^{j\theta}\notin\setL$ for
$\theta\in(\pi/2,\,\bar\theta]$. Due to \eqref{eqn:implies1} therefore the wedge 
$\W:=\{re^{j\theta}:\theta\in(\pi/2,\,\bar\theta],\,r>0\}$ and $\setL$ are disjoint, i.e.
\begin{eqnarray}\label{eqn:disjoint}
\W\cap\setL=\emptyset\,.
\end{eqnarray}
For $p\in\Natural_{\geq 1}$ let $\Gamma_{p}\in\Real^{p\times p}$ be
\begin{eqnarray*}
\Gamma_{p}:=\left[\!\!
\begin{array}{rrrrr}
-1&1&0&\cdots&0\\
0&-1&1&\cdots&0\\
\vdots&\vdots&\vdots&\ddots&\vdots\\
0&0&0&\cdots&1\\
1&0&0&\cdots&-1
\end{array}\!\!
\right]
\end{eqnarray*}
Observe for all $p$ that we have $\Gamma_{p}\in\G_{>0}$ and that
$\lambda_{p}:=e^{j\frac{2\pi}{p}}-1$ is an eigenvalue of $\Gamma_{p}$
with corresponding eigenvector
\begin{eqnarray*}
v_{p}=\left[\!\!
\begin{array}{c}
(\lambda_{p}+1)^{1}\\
(\lambda_{p}+1)^{2}\\
\vdots\\
(\lambda_{p}+1)^{p}
\end{array}
\!\!
\right]
\end{eqnarray*}
Note that for $p$ large enough $\lambda_{p}\in\W$ which contradicts
\eqref{eqn:disjoint} for $\Gamma_{p}\in\G_{>0}$. Hence the result.
\end{proof}

Below we establish Theorem~\ref{thm:main}(g).
\begin{lemma}
There exists $(C,\,A)\in\setO_{\rm F}$ that is not synchronizable with
respect to $\G_{> 0}$.
\end{lemma}

\begin{proof}
Suppose not. Then for $C=1$ and $A=1$ (note that $(C,\,A)\in\setO_{\rm F}$) there exists $L\in\Real$ such that for all
$\Gamma\in\G_{> 0}$ solutions of systems~\eqref{eqn:system} with
$u_{i}=Lz_{i}$ synchronize for all initial conditions. Choose $\varepsilon>0$
such that $1-\varepsilon L>0$. Then let
\begin{eqnarray*}
\Gamma:=\left[\!\!
\begin{array}{rr}
-\varepsilon&\varepsilon\\
0&0
\end{array}\!\!
\right] 
\end{eqnarray*}
Observe that $\Gamma\in\G_{>0}$. Consider systems~\eqref{eqn:system}
under interconnection $\Gamma$ with $u_{i}=Lz_{i}$, for $i=1,\,2$. We
can write $\dot{x}_{1}-\dot{x}_{2}=(1-\varepsilon
L)(x_{1}-x_{2})$. Therefore solutions $x_{i}(\cdot)$ do not
synchronize unless $x_{1}(0)= x_{2}(0)$. Hence the result by
contradiction.
\end{proof}

Below lemma yields Theorem~\ref{thm:main}(h).
\begin{lemma}
For no $\delta>0$ pair $(0,\,0)$ is synchronizable with
respect to $\G_{\geq \delta}$.
\end{lemma}

\section{Dual problem}\label{sec:dualproblem}

For a given interconnection $\Gamma\in\Real^{p\times p}$, consider the following 
array of $p$ identical linear systems
\begin{subeqnarray}\label{eqn:systemdual}
{\dot{x}}_{i} &=& Ax_{i}+Bu_{i} \\
 z_{i} &=& \displaystyle \sum_{j\neq i}\gamma_{ij}(x_{j}-x_{i})
\end{subeqnarray}
where $x_{i}\in\Real^{n}$, $u_{i}\in\Real^{m}$, and
$z_{i}\in\Real^{n}$ for $i=1,\,2,\,\ldots,\,p$. Matrices $A$ and $B$
are of proper dimensions. The duality between observability and
controllability for linear systems readily yields the following theorem by
which the results in this paper can be extended for the
synchronization of the arrays of coupled linear systems depicted by
\eqref{eqn:systemdual}.   

\begin{theorem}
Let $A\in\Real^{n\times n}$, $B\in\Real^{n\times m}$, and set of
interconnections $\setS$ be such that $(B^{T},\,A^{T})$ is
synchronizable with respect to $\setS$. Then and only then there
exists a linear feedback law $K\in\Real^{m\times n}$ such that for
each $\Gamma\in\setS$ solutions $x_{i}(\cdot)$ of
array~\eqref{eqn:systemdual} with $u_{i}=Kz_{i}$ synchronize for all
initial conditions.
\end{theorem}


\section{Conclusion}\label{sec:conclusion}
For arrays of identical output-coupled linear systems we have
investigated the sufficiency of certain conditions on system matrix
$A$ and output matrix $C$ for existence of a feedback
law under which the systems synchronize for all coupling
configurations with connected graphs. Namely, we have filled in the
previously missing pieces (the boxes indicated with a question mark)
of the chart given in Fig.~\ref{fig:chart}. In addition, for each
case corresponding to a box with ``\checkmark'' we have designed a
synchronizing feedback law.
\begin{figure}[h]
\begin{center}
\includegraphics[scale=0.5]{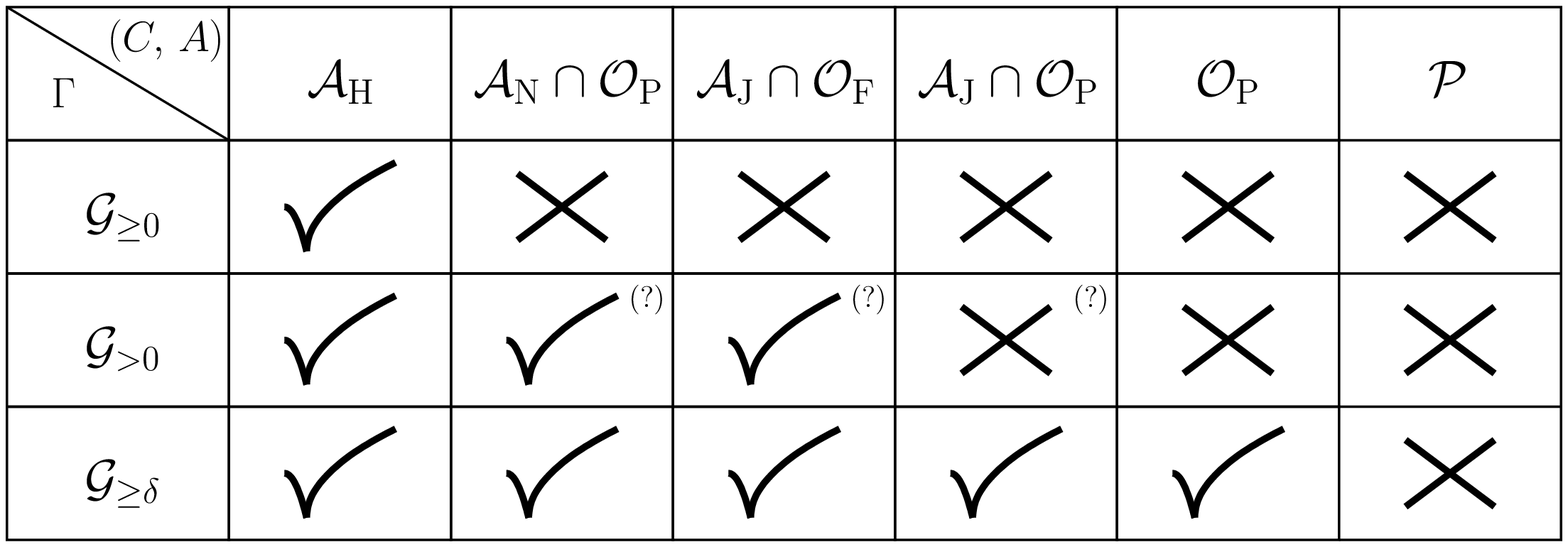}
\caption{Sufficiency of certain conditions on $(C,\,A)$ for synchronizability  
with respect to different sets of interconnections
$\Gamma$.}\label{fig:chart}
\end{center}
\end{figure}

\bibliographystyle{plain}
\bibliography{references} 
\end{document}